\documentclass[a4,12pt]{amsart}
\usepackage{amssymb,amstext,amsmath,amscd,amsthm,amsfonts,enumerate,latexsym}
\usepackage{color}
\usepackage[dvipdfmx]{graphicx}
\usepackage[all]{xy}
\theoremstyle{plain}
\newtheorem{thm}{Theorem}[section]

\newtheorem*{thm*}{Theorem}
\newtheorem*{cor*}{Corollary}

\newtheorem{prop}[thm]{Proposition}

\newtheorem{lem}[thm]{Lemma}
\newtheorem{cor}[thm]{Corollary}

\newtheorem*{claim*}{Claim}

\theoremstyle{definition}

\newtheorem{ex}[thm]{Example}

\theoremstyle{remark}

\numberwithin{equation}{thm}

\def\GCD{\operatorname{GCD}}

\def\Max{\operatorname{Max}}

\def\mod{\mathrm{mod}}

\def\m{\mathfrak m}

\newcommand{\Aut}{\mathrm{Aut}}

\newcommand{\calC}{\mathcal{C}}

\newcommand{\fka}{\mathfrak{a}}
\newcommand{\fkb}{\mathfrak{b}}

\newcommand{\fkp}{\mathfrak{p}}

\newcommand{\mapright}[1]{%
\smash{\mathop{%
\hbox to 1cm{\rightarrowfill}}\limits^{#1}}}

\newcommand{\mapleft}[1]{%
\smash{\mathop{%
\hbox to 1cm{\leftarrowfill}}\limits_{#1}}}

\def\Supp{\operatorname{Supp}}
\def\Spec{\operatorname{Spec}}

\title[Efficient generation of ideals]{Efficient generation of ideals in core subalgebras of the polynomial ring $k[t]$ over a field $k$}

\author{Naoki Endo}
\address{Global Education Center, Waseda University, 1-6-1 Nishi-Waseda, Shinjuku-ku, Tokyo 169-8050, Japan}
\email{naoki.taniguchi@aoni.waseda.jp}
\urladdr{http://www.aoni.waseda.jp/naoki.taniguchi/}

\author{Shiro Goto}
\address{Department of Mathematics, School of Science and Technology, Meiji University, 1-1-1 Higashi-mita, Tama-ku, Kawasaki 214-8571, Japan}
\email{shirogoto@gmail.com}

\author{Naoyuki Matsuoka}
\address{Department of Mathematics, School of Science and Technology, Meiji University, 1-1-1 Higashi-mita, Tama-ku, Kawasaki 214-8571, Japan}
\email{naomatsu@meiji.ac.jp}

\author{Yuki Yamamoto}
\address{Department of Mathematics, School of Science and Technology, Meiji University, 1-1-1 Higashi-mita, Tama-ku, Kawasaki 214-8571, Japan}
\email{yuki.yamamoto5104@gmail.com}


\thanks{The first author was partially supported by JSPS Grant-in-Aid for Young Scientists (B) 17K14176 and Waseda University Grant for Special Research Projects 2019C-444, 2019E-110. The second author was partially supported by JSPS Grant-in-Aid for Scientific Research (C) 16K05112. The third author was partially supported by JSPS Grant-in-Aid for Scientific Research (C) 18K03227.}

\begin{document}

\maketitle

\setlength{\baselineskip} {14.9pt}

\begin{abstract}
This note aims at finding explicit and efficient generation of ideals in subalgebras $R$ of the polynomial ring $S=k[t]$ ($k$ a field) such that $t^{c_0}S \subseteq R$ for some integer $c_0 > 0$. The class of these subalgebras which we call cores of $S$ includes the semigroup rings $k[H]$ of numerical semigroups $H$, but  much larger than the class of numerical semigroup rings. For $R=k[H]$ and $M \in \Max R$, our result eventually shows that $\mu_{R}(M) \in \{1,2,\mu(H)\}$ where $\mu_{R}(M)$ (resp. $\mu(H)$) stands for the minimal number of generators of $M$ (resp. $H$), which covers in the specific case the classical result of O. Forster--R. G. Swan.
\end{abstract}


\section{Introduction}
This note aims at finding efficient systems of generators for ideals in certain subalgebras $R$ of the polynomial ring $S=k[t]$ with one indeterminate $t$ over a field $k$. The class of subalgebras which this note concerns naturally includes the semigroup rings $k[H]$ of numerical semigroups $H$.

Investigation on the numbers of generators of ideals and modules is one of the classical subjects of great interest in commutative algebra. One of the main problems was to look for a bound,  in terms of the local data, on the minimal number of generators for a finitely generated module $M$ over a commutative Noetherian ring $R$. This problem was solved by O. Forster \cite{F} in 1964, and subsequently in 1967, R. G. Swan \cite{Swan} gave an efficient bound for the number of generators, also generalizing Forster's argument to the non-commutative case. The reader may consult \cite{EE2, EE}, where D. Eisenbud and E. G. Evans, Jr. extended various stability theorems for projective modules to the context of arbitrary finitely generated modules, developing a beautiful theory of basic elements. Let us note one of the results in the following form. Throughout, let $\mu_R(*)$ stand for the minimal number of generators.

\begin{thm}[Forster-Swan Theorem on the number of generators of a module, {\cite[Corollary 3]{EE2}, \cite[Corollary 5]{EE}}]\label{fs}
Let $M$ be a finitely generated module over a finitely generated algebra $R$ over a field and set $$\operatorname{b}(M) = \sup_{\fkp \in \Spec R}[\dim R/\fkp + \mu_{R_\fkp}(M_\fkp)].$$ Then $\mu_R(M) \le \operatorname{b}(M)$.
\end{thm}

On the other hand, in 1978 J. Sally \cite{S} explored the case where the base rings are local, giving several fundamental results about the numbers of generators of ideals in Cohen-Macaulay local rings of small dimension. Combining her results with Theorem \ref{fs}, we nowadays have satisfactorily general methods to estimate the numbers of generators for the ideals in rings of small dimension. For example, if $R=k[H] = \sum_{h \in H}kt^h$ is the semigroup ring of a numerical semigroup $H$ where $t$ denotes an indeterminate over a field $k$, then $\Spec R$ has at most an isolated singularity, so that for a maximal ideal $M$ of $R$ we get $\mu_R(M) \le 2$ if $M \ne P_0$ and $\mu_R(P_0)=\mu(H)$, where $P_0=(t^h \mid 0 < h \in H)$ and $\mu(H)$ stands for the minimal number of generators of $H$. Nevertheless, even for the maximal ideals $M$ of $k[H]$ it could be another problem to look for explicit and efficient systems of generators, which we shall pursue in this note.

We explain how this note is organized. In Section \ref{section2} we study $k$-subalgebras $R$ of the polynomial ring $S = k[t]$ over a field $k$ such that $t^{c_0} S \subseteq R$ for some integer $c_0 > 0$. The class of these algebras which we call cores of $S$ naturally includes (but much larger than) the class of numerical semigroup rings. We consider the ideals $I$ of the form $I = fS \cap R$ where $f \in S=k[t]$ such that $f (0) =1$, and show that $I$ is at most $2$-generated, giving explicit generators of $I$ which depend only on $c_0$ and $f$ (Theorem \ref{3.2}). The result leads to the study of the integral closures $\overline{I}$ of arbitrary ideals $I$ of $R$, and in Section 3, we will give an estimation of the minimal number of generators of $\overline{I}$. In Section 4, we shall focus on the case where $R = k[H]$, giving rather incredible systems of generators of $k$-rational closed points $M$ of $\Spec R$ (Theorem \ref{4.2}, Corollary \ref{4.3}), which eventually shows that $\mu_R(M) \in \{2, \mu(H)\}$, if $1 \not\in H$ and the field $k$ is algebraically closed (Corollary \ref{4.4}).

\section{Main result}\label{section2}

Let $S = k[t]$ be the polynomial ring over a field $k$. Let $0 < a_1, a_2, \cdots, a_v \in \Bbb Z$~($v > 0$) be integers such that $\GCD (a_1, a_2, \cdots, a_v)=1$. We denote by $$H = \left<a_1, a_2, \cdots, a_v \right>=\left\{\sum_{i=1}^vc_ia_i \mid 0 \le c_i \in \Bbb Z~~\text{for}~~1 \le \forall i \le v \right\}$$ the numerical semigroup generated by $a_i's$. The reader may consult \cite{RG} for a general reference about numerical semigroups. We set $k[H] = k[t^{a_1}, t^{a_2}, \cdots, t^{a_v}]\subseteq S$ and call it the semigroup ring of $H$. Let
$$\operatorname{c}(H)= \min \{n \in \Bbb Z \mid m \in H~\text{for~all}~m \in \Bbb Z~\text{such~that}~m \ge n\}$$ denote the conductor of $H$. Hence, $k[H]:S = t^{\operatorname{c}(H)}S$. Unless otherwise specified, we throughout assume that $1 \not\in H$. Therefore, $\operatorname{c}(H) \ge 2$, and  $\operatorname{Sing}(k[H]) = \{P_0\}$, where $P_0 =(t^{a_1}, t^{a_2}, \cdots, t^{a_v})=tS \cap k[H]$

Let $R$ be a $k$-subalgebra of $S$. Then we say that $R$ is a {\em core} of $S$, if $t^{c_0}S \subseteq R$ for some integer $c_0 > 0$. If $R$ is a core of $S$, then $$k[t^{c_0}, t^{c_0+1}, \ldots, t^{2c_0 -1}] \subseteq R \subseteq S,$$ and a given $k$-subalgebra $R$ of $S$ is a core of $S$ if and only if $R \supseteq k[H]$ for some numerical semigroup $H$. Therefore, once $R$ is a core of $S$, $R$ is a finitely generated $k$-algebra of dimension one, and $S$ is a birational module-finite extension of $R$ with $t^{c_0}S \subseteq R:S$. Typical examples of cores are, of course, the semigroup rings $k[H]$ of numerical semigroups $H$. However, cores of $S$ do not necessarily arise as semigroup rings for some numerical semigroups. Let us note one of the simplest examples.

\begin{ex}
Let $R = k[t^2+t^3] + t^4S$. Then $R \ne k[H]$ for any numerical semigroup $H$.
\end{ex}

We summarize a few basic facts about cores.
For $P \in \Spec R$, we say that $P$ is a {\it $k$-rational closed point} of $\Spec R$, if $k = R/P$.

\begin{prop}\label{core}
Let $R$ be a core of $S$ and set $P_0 = tS \cap R$.  Then the following assertions hold true.
\begin{enumerate}[${\rm (1)}$]
\item[${\rm (1)}$] Let $Q \in \Spec S$ and set $P = Q \cap R$. If $P \ne P_0$, then $R_P = S_P = S_Q$.
\item[${\rm (2)}$] The map $\Spec S \to \Spec R, Q \mapsto Q\cap R$ is a bijection.
\item[${\rm (3)}$] Let $Q \in \Spec S$ and set $P = Q \cap R$. Then $Q$ is a $k$-rational closed point of $\Spec S$ if and only if so is $P$ in $\Spec R$.
\end{enumerate}
\end{prop}

\begin{proof}
(1) If $t^{c_0}S \subseteq P$, then $t \in Q$, so that $P = P_0$. Therefore, $t^{c_0}S \not\subseteq P$, whence $R_P = S_P$, because $R:S \not\subseteq P$. Consequently, $S_P = S_Q$, since $S_P$ is a local ring and $QS_P \in \Max S_P$.

(2) Since $S$ is integral over $R$, the map $\Spec S \to \Spec R, Q \mapsto Q\cap R$ is surjective. Let $Q_1, Q_2 \in \Spec S$ and assume that $Q_1\cap R = Q_2 \cap R = P$. We will show that $Q_1 = Q_2$. To do this, we may assume that $P \ne (0)$. If $P \ne P_0$, then $Q_1= Q_2$, because by assertion (1) $S_P$ is a local ring and $Q_1S_P, Q_2S_P \in \Max S_P$. If $P=P_0$, then $t^{c_0} \in Q_1 \cap Q_2$, so that $Q_1 = Q_2 = tS$, which proves assertion (2).

(3) We may assume that $P \in \Max R$, $Q \in \Max S$, and by assertion (2) that $P \ne P_0$. Therefore, $R/P = R_P/PR_P$ and $S/Q = S_Q/QS_Q$, whence $R/P = S/Q$, because $R_P = S_Q$ by assertion (1). Thus, $k = R/P$ if and only if $k=S/Q$.
\end{proof}

In what follows, we fix a core $R$ of $S= k[t]$. Let $f \in S=k[t]$ such that $f(0) =1$. Choose integers $\ell, c \ge c_0$ so that $\ell \ge 2$. We write $f = \sum_{i \ge 0}a_it^i$ with $a_i \in k$ (hence $a_0=1$) and consider the following $(\ell -1) \times \ell$ matrix


$$A=\left[\begin{smallmatrix}
a_1       &  1       & 0           & 0      &\cdots  &   \cdots & 0     \\
a_2       &a_1       & 1           & 0      &\cdots  &   \cdots & 0     \\
a_3       &a_2       & a_1         & 1      &   0    &   \cdots & 0     \\
 \vdots   &  \vdots  & \vdots      & \vdots &\vdots  & \vdots   & \vdots\\
a_{\ell-2}&a_{\ell-3}&a_{\ell - 4} & \cdots & a_1    & 1        & 0\\
a_{\ell-1}&a_{\ell-2}&a_{\ell - 3} & \cdots & a_2    & a_1      & 1
\end{smallmatrix}\right].$$


\noindent
Then, $\operatorname{rank}A= \ell - 1$, so that there exists a unique element $\mathbf{v} = \left[\begin{smallmatrix}
v_0\\
v_1\\
v_2\\
 \vdots\\
v_{\ell-1}
\end{smallmatrix}\right] \in k^{(\ell)}$
such that $A\mathbf{v}=\mathbf{0}$ and $v_0 = 1$. Set $g = \sum_{i=0}^{\ell -1}v_it^i \in S$. Hence, $g(0)=1$, and we have the following.

\begin{prop}\label{3.1}
The following assertions hold true.
\begin{enumerate}[${\rm (1)}$]
\item[${\rm (1)}$] $fg -1 \in t^\ell S$, whence $fg \in R$.
\item[${\rm (2)}$]
$f \in R$ if and only if $g \in R$.
\item[${\rm (3)}$] Let $\varphi \in \sum_{i=0}^{\ell -1}kt^i$. If $f\varphi -1 \in t^\ell S$, then $\varphi = g$.
\end{enumerate}
\end{prop}

\begin{proof}
(1) The coefficient $b_n$ of $t^n$ in $fg$ is given by $b_n = \sum_{i=0}^na_iv_{n-i}$, so that $b_n = 0$ for all $1 \le n \le \ell-1$, while $(fg)(0)=f(0){\cdot}v_0=1$. Therefore $$fg -1 \in t^{\ell}S.$$

(2) Let $\overline{R}= R/I$ and $\overline{S} = S/I$, where $I = t^{\ell}S$. Then, $\overline{R}$ is a subring of $\overline{S}$, and $\overline{S}$ is a module-finite extension of $\overline{R}$. Therefore, $\overline{R}$ is a local ring, since so is $\overline{S}$, and $\m_{\overline{R}}=\m_{\overline{S}} \cap \overline{R}$, where $\m_{\overline{R}}$  and $\m_{\overline{S}}$ denote respectively  the maximal ideals of $\overline{R}$ and $\overline{S}$. Then, since $fg \equiv 1 ~\mod~t^\ell S$ in $S$ and $t^\ell S = I$, we have $\overline{f}\overline{g} = 1$ in $\overline{S}$, where $\overline{f}, \overline{g}$ denote respectively the images of $f,g$ in $\overline{S}$. Therefore, if $f \in R$, $\overline{f} \in \overline{R}$ and it is a unit of $\overline{R}$, since it is a unit of $\overline{S}$. Because $\overline{g}$ is the inverse of $\overline{f}$ in the ring $\overline{S}$, which should belong to the ring $\overline{R}$, just thanks to the uniqueness of the inverse. Therefore, $g \in R$, since $g \equiv r ~\mod~I$ in $S$ for some $r \in R$. The converse is similarly proved.

(3) This is clear, since $\overline{\varphi} = \overline{g}$ in $\overline{S}$.
\end{proof}

We set $I = fS \cap R$. Then, $fg \in I$ since $fg \in R$, and $t^cf \in I$ for  every integer $c \ge c_0$ since $t^c S \subseteq R$. Therefore, $(t^cf, fg) \subseteq I$. We furthermore have the following.

\begin{thm}\label{3.2} The  following assertions hold true.
\begin{enumerate}[${\rm (1)}$]
\item[${\rm (1)}$] $I = (t^c f, fg)$ for every integer $c \ge c_0$, and $IS = fS$.
\item[${\rm (2)}$] $I$ is a principal ideal of $R$ if and only if $f \in R$. When this is the case, $I= fR$.
\end{enumerate}
\end{thm}

Let us divide the proof of Theorem \ref{3.2} into several steps.  We may assume that $f \not\in k$. We fix an irreducible decomposition $f = \prod_{i=1}^nf_i^{e_i}$ of $f$. Hence, ${f_i}'s$ are irreducible polynomials such that $f_i S = f_jS$ only if $i=j$, and ${e_i}'s$ are positive integers. We set $\Lambda = \{1,2,\ldots, n\}$. Let $$Q_i = f_i S, \  P_i = Q_i \cap R~\text{for~each}~i \in \Lambda,\  \text{and}~ P_0 =tS \cap R.$$ Then, $fg \not\in P_0$, since $(fg)(0)= f(0)g(0)=1$, but $fg \in P_i$ for every $i \in \Lambda$, since $f \in Q_i$. Let $P \in \Spec R$ and write $P = Q \cap R$ for some $Q \in \Spec S$. Let $c \ge c_0$ be an integer. Then, $t^cf \in P$ if and only if $t^c f \in Q$, and the latter condition is equivalent to saying that either $t \in Q$ or $f_i \in Q$ for some $i \in \Lambda$. Therefore, setting $\operatorname{V}(t^cf) = \{P \in \Spec R \mid t^cf \in P\}$, we have the following.

\begin{prop}\label{claim1}
The following assertions hold true.
\begin{enumerate}[${\rm (1)}$]
\item[${\rm (1)}$] $fg \in P_i$ for each $i \in \Lambda$, but $fg \not\in P_0$. Consequently, $IR_{P_0} = R_{P_0}$.
\item[${\rm (2)}$] $\operatorname{V}(t^cf) = \{P_0, P_1, \ldots, P_n\}$.
\end{enumerate}
\end{prop}

\begin{prop}\label{claim2}
For $i \in \Lambda$ the following assertions hold true.
\begin{enumerate}[${\rm (1)}$]
\item[${\rm (1)}$] $R_{P_i} =S_{Q_i}$ and $IR_{P_i}=fR_{P_i}$, whence $I \subseteq P_i$.
\item[${\rm (2)}$] Let $h \in S$. If  $h \not\in Q_i$, then $IR_{P_i} = fhR_{P_i}$.
\end{enumerate}
Therefore, for each $P \in \Spec R$, $I \subseteq P$ if and only if $P = P_j$ for some $j \in \Lambda$.
\end{prop}

\begin{proof}
By Proposition \ref{core} (1) $R_{P_i}= S_{P_i} = S_{Q_i}$, since $P_i \ne P_0$, while $$IR_{P_i}= [fS \cap R]R_{P_i} = fS_{P_i} \cap R_{P_i}.$$ Therefore $IR_{P_i}= fS_{Q_i} \cap S_{Q_i}=fS_{Q_i}=fR_{P_i}$, so that $I \subseteq P_i$ because $fS_{Q_i} \ne S_{Q_i}$. If $h \in S$ but $h  \not\in Q_i$, then $IR_{P_i} = fS_{Q_i}=fhS_{Q_i}=fhR_{P_i}$. Let $P \in \Spec R$. If $I \subseteq P$, then $t^c f \in P$, so that by Proposition \ref{claim1} $P =P_j$ for some $j \in \Lambda$. Since by assertion (1) $I \subseteq P_i$ for every $i \in \Lambda$, the last assertion follows.
\end{proof}

The following is a direct consequence of Proposition \ref{claim2}.

\begin{cor}\label{cor}
$I$ is an invertible ideal of $R$.
\end{cor}

We set $W = R \setminus \bigcup_{i=0}^nP_i$, Then, $W^{-1}I$ is a principal ideal of $W^{-1}R$, because it is an invertible ideal of $W^{-1}R$ (Corollary \ref{cor}) and the ring  $W^{-1}R$ is a one-dimensional semi-local ring whose maximal ideals are precisely $\{W^{-1}P_i\}_{0 \le i \le n}$. We now notice that $R/(t^cf) = W^{-1}[R/(t^cf)] $, since for every $w \in W$ the image of $w$ in $R/(t^cf)$ is invertible in $R/(t^cf)$ (Proposition \ref{claim1} (2)). Therefore, in order to prove Theorem \ref{3.2} (1), it suffices to show that $W^{-1}I = fg{\cdot}W^{-1}R$, or equivalently $$\Supp_{W^{-1}R}W^{-1}I/fg{\cdot}W^{-1}R = \emptyset.$$ Because $\Max W^{-1}R= \{W^{-1}P_i \mid  0 \le i \le n\}$, by Proposition \ref{claim1} (1) and Proposition \ref{claim2} (2) this is certainly the case, once $g \not\in Q_i$ for any $i \in \Lambda$. Suppose that $g \in Q_i$ for some $i \in \Lambda$, and set $\xi = \prod_{j \in \Gamma}f_j$, where $\Gamma = \{j \in \Lambda \mid g \not\in Q_j\}$. Choose an integer $q$ so that $q \ge c + c_0$ and set  $h = g + t^q\xi$. We then have the following.

\begin{prop}\label{claim3} We have $h \not\in Q_i$ for any $i \in \Lambda$. Hence, $IR_{P_i} = fhR_{P_i}$ for every $i \in \Lambda$.
\end{prop}

\begin{proof}
Assume the contrary and let $h \in Q_i$ for some $i \in \Lambda$. If $g \in Q_i$,  we then have $t^q \xi \in Q_i$, so that either $t \in Q_i$ or $\xi \in Q_i$. However, if $t \in Q_i$, then $Q_i = tS$, which forces $f(0)=0$ because $f \in Q_i$. Therefore, $\xi \in Q_i$, so that $f_j \in Q_i$ for some $j \in \Gamma$. Hence, $i=j \in \Gamma$, whence $g \not\in Q_i$. This is a contradiction. Thus, $g \not\in Q_i$, that is $i \in \Gamma$, whence $t^q\xi\in Q_i$ so that $h=g + t^q \xi \not\in Q_i$. This is also a contradiction. Hence, $h \not\in Q_i$ for any $i \in \Lambda$. The second assertion is a direct consequence of Proposition \ref{claim2} (2).
\end{proof}

Since $t^{q-c}\xi \in R$ (remember that $q \ge c + c_0$), we have $fh = fg + t^cf{\cdot}t^{q-c}\xi \equiv fg~\mod~t^cfR$, so that $(t^cf, fh) = (t^cf, fg)$ in $R$. On the other hand, because $fh \not\in P_0$ (notice that $(fh)(0)=(fg)(0)=1$) and by Proposition \ref{claim3} $IR_{P_i}= fhR_{P_i}$ for every $i \in \Lambda$, we get $\Supp_{W^{-1}R}W^{-1}I/fh{\cdot}W^{-1}R = \emptyset$, so that $I = (t^cf, fh)$. Therefore, $I = (t^cf,fg)$. Because $(t^c, g)S=S$ and $IS =(t^cf,fg)S= f{\cdot}(t^c, g)S$, we readily get $IS=fS$, which proves assertion (1) of Theorem \ref{3.2}.

Let us consider assertion (2). Suppose that $I=(t^cf, fg)$ is a principal ideal of $R$ and let  $I = \varphi R$ for some $\varphi \in R$. We write $\varphi =f \psi$ with $\psi \in S$. Then, since $\psi R = t^cR+ gR$ and $t^cS + gS=S$, we have $\psi S=S$, so that $0 \ne \psi \in k$. Therefore, $I=\varphi R=f\psi R = fR$, whence $f \in R$. If $f \in R$, then $fR \subseteq I = fS \cap R$, while $g \in R$ by Proposition \ref{3.1}. Consequently, because $t^c, g \in R$, we get $fR \subseteq I = (t^cf, fg) \subseteq fR$. Hence, $I = fR$, which completes the proof of Theorem \ref{3.2}.

\begin{ex}
Let $f = 1-t$. Then, $g = \sum_{i=0}^{\ell-1} t^i$, where $\ell \ge \max\{2, c_0\}$. We set $I = (1-t)S \cap R$. Then, $I$ is a maximal ideal of $R$, and $I=(t^c-t^{c+1}, 1-t^\ell)$ for every $c \ge c_0$. The ideal $I$ is a principal ideal of $R$ if and only if $R=S$.
\end{ex}

\begin{ex}
Let $k = \Bbb Z/(2)$ and $f = 1 + t^2+t^3+t^5+t^6 \in S = k[t]$. Then $f$ is an irreducible polynomial in $S$. Choose a $k$-subalgebra $R$ of $S$ so that $t^{10}S \subseteq R$. Let $I = fS \cap R$ and set $\ell = c = 10$. We then have $g = 1 + t^2+t^3+t^4+t^5+t^6+t^7$ and $I = (t^{10}f, 1+t^{10}+t^{13})$. The maximal ideal $I$ is a principal ideal of $R$ if and only if $f \in R$.
\end{ex}

We consider semigroup rings $R = k[H]$ of numerical semigroups $H$.

\begin{ex}
Let $e \ge 2$ be an integer and set $H = \left<e, e + 1, \cdots, 2e -1\right>$. Hence, $R = k[t^{e +i} \mid 0 \le i \le e-1]$ and $\operatorname{c}(H) = e$. Let $0 \ne \alpha \in k$ and set $f = 1-\alpha t \in S$, $M = fS \cap R$. Then, taking $\ell = c = e$, we have $g = \sum_{i=0}^{e-1} \alpha^it^i$, whence
$$M = (t^{e} -\alpha t^{e +1}, 1-\alpha^e t^e) =\left( \textstyle \left(\frac{1}{\alpha}\right)^{e + 1}-t^{e + 1}, \left(\frac{1}{\alpha}\right)^e -t^e\right).$$
A similar result holds true for $k$-rational closed points except the origin of arbitrary monomial curves $\Spec k[H]$, which we shall discuss in Section 4.
\end{ex}

\begin{ex}\label{3.3}
Let $k = \Bbb R$ and $f = 1 + at + bt^2$, where $a,b \in k$ such that $b \ne 0$ and $a^2 < 4b$. Let $H = \left<2,5\right>$ or $H = \left<4,5,6,7\right>$. Hence $\operatorname{c}(H) = 4$. Let $R = k[H]$. We set $M = fS \cap R$ and choose $\ell = c = 4$. Then
$$
A = \left[\begin{smallmatrix}
a&1&0&0\\
b&a&1&0\\
0&b&a&1
\end{smallmatrix}
\right], \ \ \ \mathbf{v} = \left[\begin{smallmatrix}
1 \\
-a\\
a^2-b\\
2ab-a^3\\
\end{smallmatrix}\right].
$$
We have $fg = 1+ (3a^2b-b^2-a^4)t^4 + (2ab-a^3)bt^5$ and $M = (t^4f,fg)$. We now take $a =0$. Then $g = 1-bt^2$ and $M = (t^4f,1-b^2t^4)$. By Theorem \ref{3.2} (2), we have $M = fR$ if $H =\left<2,5\right>$, but $\mu_R(M) = 2$ if $H = \left<4,5,6,7\right>$. This example shows that even though the generating system of $M = fS \cap R$ depends only on $\operatorname{c}(H)$ and $f$, the whole structure of $H$ has an influence on the minimal number $\mu_R(M)$ of generators for $M$.
\end{ex}

\section{Integral closures $\overline{I}$ of an ideal $I$ in $R$}\label{section3}
Similarly as in Section \ref{section2},  we fix a core $R$ of $S$. Hence, $R$ is a $k$-subalgebra $R$ of $S=k[t]$ such that $t^{c_0}S \subseteq R$ for some integer $c_0 > 0$. Let $I ~(\ne (0))$ be an ideal of $R$. We write $IS = \varphi S$ with $\varphi \in S$. In this section, we are interested in the efficient generation of the integral closure $\overline{I}$ of $I$. To do this, notice that $\overline{I}= IS \cap R$, since $S$ is a module-finite (hence an integral) extension of $R$. Without loss of generality, we may assume that $\varphi = t^q f$, where $q \ge 0$ is an integer and $f \in S$ such that $f(0)=1$. We choose integers $\ell, c \ge c_0$ so that $\ell \ge 2$ in order to obtain the polynomial $g \in S$ explored in Section \ref{section2} (see Proposition \ref{3.1}). Hence $fS \cap R = (t^cf,fg)$ by Theorem \ref{3.2}.

With this notation we have the following.

\begin{lem}\label{lem1}
$\overline{I}= (t^q S \cap R)\cap (fS \cap R) = (t^qS \cap R){\cdot}(fS \cap R)$ and $(t^qS \cap R)S = t^qS$.
\end{lem}

\begin{proof}
We have $t^qfS = t^qS \cap fS$, because $(t^q, f)S= S$, so that the first equality follows, since $\overline{I} = (t^qfS) \cap R$. To see the second equality, it suffices to show $(t^qS \cap R)+(fS \cap R)=R$. Assume the contrary and choose $M \in \Max R$ so that $(t^qS \cap R)+(fS \cap R) \subseteq M$. We write $M = N \cap R$ for some $N \in \Max S$. Then, since $t^{c_0 + q}S \subseteq R$, we get $t^{c_0+q} \in M$, whence $t \in N$. On the other hand, we have $fg \in N$, since $fg \in fS\cap R$ (Proposition \ref{3.1} (1)). Therefore, $t, fg \in N$, which is impossible, because $(fg)(0)=1$. Hence, $(t^qS \cap R)+(fS \cap R)=R$, and the second equality follows.  To see that $(t^qS \cap R)S = t^qS$, notice that $IS =\overline{I}S$, since $$IS \subseteq \overline{I}S \subseteq \overline{IS} = \overline{\varphi S} = \varphi S = IS.$$ We then have $t^qfS = IS=\overline{I}S= [(t^q S \cap R){\cdot}(fS \cap R)]S= [(t^qS \cap R)S]{\cdot}[(fS \cap R)S]$, so that $t^qS=(t^qS \cap R)S$, because $(fS\cap R)S=fS$ by Theorem \ref{3.2}.
\end{proof}

We furthermore have the following.

\begin{prop}\label{prop1}
The following assertions hold true.
\begin{enumerate}[${\rm (1)}$]
\item[${\rm (1)}$] If $c \ge q$, then $\overline{I}= (t^cf)+ fg{\cdot}(t^qS \cap R)$, whence $\mu_R(\overline{I}) \le 1 + \mu_R(t^qS \cap R)$.
\item[${\rm (2)}$] If $q \ge c_0$, then $\overline{I} = \varphi S$ and $\mu_R(\overline{I}) = \mu_R(S)$.
\end{enumerate}
\end{prop}

\begin{proof}
(1) Since $t^cf \in t^qS \cap R$ and $fS \cap R = (t^cf, fg)$ by Theorem \ref{3.2}, we have by Lemma \ref{lem1} that $$\overline{I}= (t^q S \cap R)\cap (fS \cap R)= t^cfR + [(t^qS \cap R) \cap (fg)],$$ while $(t^qS \cap R) \cap (fg)= (fg){\cdot}(t^qS \cap R)$, since $(t^qS \cap R)+ (fg) = R$ (see the proof of the second equality in Lemma \ref{lem1}). Thus, $\overline{I}= (t^cf)+ fg{\cdot}(t^qS \cap R)$.

(2) We have $\overline{I}= \varphi S \cap R = \varphi S$, since $\varphi \in t^{c_0}S \subseteq R$.
\end{proof}

We consider the case where $R=k[H]$. Let $e \ge 2$ be an integer and set $R=k[H]$, where $H=\left<e, e+1, \ldots, 2e-1\right>$. Then, since $(t^qS \cap R)S=t^qS$ by Lemma \ref{lem1}, we have $q \in H$. Therefore, either $q=0$, or $q \ge e = c_0$, so that Proposition \ref{prop1} shows the following, since $\mu_R(S)=e$.

\begin{cor}
Let $e \ge 2$ be an integer and set $H=\left<e, e+1, \ldots, 2e-1\right>$. Let $R = k[H]$. Then, for each ideal $I~(\ne (0))$ of $R$, we have $\mu_R(\overline{I}) \in \{1,2, e\}$.
\end{cor}

\section{Maximal ideals of numerical semigroup rings}
In this section we study the semigroup rings of numerical semigroups.  In what follows, let $0 < a_1, a_2, \cdots, a_v \in \Bbb Z$ be integers such that $\GCD (a_1, a_2, \cdots, a_v)=1$. Let $H = \left<a_1, a_2, \cdots, a_v \right>$, and $S= k[t]$, where $k$ is a field. We consider the ring $$R=k[H] = k[t^{a_1}, t^{a_2}, \cdots, t^{a_v}],$$ assuming that $1 \not\in H$. Therefore, $\operatorname{c}(H) \ge 2$, and  $\operatorname{Sing}(R) = \{P_0\}$, where $P_0 =(t^{a_1}, t^{a_2}, \cdots, t^{a_v})=tS \cap R.$

Let $M \in \Spec R$. Recall that $M$ is said to be a {\it $k$-rational closed point} of $\Spec R$, if $k = R/M$. As for the rationality in $\Spec R$, the following result is well-known, which allows us to naturally identify the $k$-rational closed points of $\Spec R$ with the points of the monomial curve $\calC =\{(\alpha^{a_1},\alpha^{a_2}, \ldots, \alpha^{a_v}) \mid \alpha \in k\}$.

\begin{prop}\label{2.3}
Let $M \in \Max R$. Then the following conditions are equivalent.
\begin{enumerate}[${\rm (1)}$]
\item[${\rm (1)}$] $M$ is a $k$-rational closed point of $\Spec R$.
\item[${\rm (2)}$] $M = (\alpha^{a_i}- t^{a_i} \mid 1 \le i \le v)$ for some $\alpha \in k$.
\item[${\rm (3)}$] $M \subseteq (\alpha - t)S$ for some $\alpha \in k$.
\end{enumerate}
When this is the case, the element $\alpha \in k$ given in conditions $(2)$ and $(3)$ is uniquely determined for $M$.
\end{prop}

Let us consider $k$-rational closed points of $\Spec R$ which do not correspond to the origin of the curve $\calC =\{(\alpha^{a_1},\alpha^{a_2}, \ldots, \alpha^{a_v}) \mid \alpha \in k\}$. We begin with the following.

\begin{thm}\label{4.2}
Let $M= (1 - t^{a_i} \mid 1 \le i \le v)$. Then $$M = (1 - t^c \mid c \in H) = (1 -t^a, 1 - t^b)$$
for all $0 < a, b \in H$ such that $\operatorname{GCD}(a,b)=1$.
\end{thm}

\begin{proof}  Let $0 < c \in H$. Without loss of generality, we may assume that $c = (-n)a + mb$ with $n,m \ge 2$. Then, since $$\sum_{q=1}^n\left[1 - t^{(q-1)a+c}\right]= \sum_{q=1}^n\left[(1-t^{qa+c}) - (1-t^a)t^{(q-1)a + c}\right],$$ we have
$$1-t^c= (1-t^{na+c})+(1-t^a){\cdot}\sum_{q=1}^{n}t^{(q-1)a +c}.$$ Consequently, because $1-t^{na+c} = 1-t^{mb} = (1-t^b)\sum_{q=0}^{m-1}(t^b)^q \in (1-t^b)$, we get $1-t^c \in (1-t^a, 1-t^b)$. Therefore
$$M \subseteq (1-t^c \mid c \in H) \subseteq (1-t^a, 1-t^b) \subseteq (1-t)S \cap R,$$
whence the required equalities follow, because $M=(1-t)S \cap R$  by Proposition \ref{2.3}.
\end{proof}

\begin{ex}
Let $H=\left<3,5,7\right>$ and set $M = (1-t^c \mid c \in H)$. Then, $M$ is a maximal ideal of $R=k[t^3, t^5,t^7]$ and $M = (1-t^3,1-t^5)=(1-t^3,1-t^7)=(1-t^5,1-t^7)$.
\end{ex}


Let $0 \ne \alpha \in k$ and set $M_\alpha=(\alpha^{a_i} - t^{a_i} \mid 1 \le i \le v)$ in $R=k[H]$. Let $\varphi_\alpha \in \Aut_kk[t]$ defined by $\varphi_\alpha (t) = t_1$, where $t_1 = \frac{t}{\alpha}$. Then, $\varphi_\alpha(k[H])=k[H]$, so that $\varphi_\alpha$ induces the automorphism $\psi_\alpha$ of $R$. Therefore, since $\psi_\alpha(M_1)=M_\alpha$, we get the following.

\begin{cor}\label{4.3}
Let $0 \ne \alpha \in k$. Then $M_\alpha = (\alpha^a - t^a, \alpha^b -t^b)$ for all $0 < a, b \in H$ such that $\operatorname{GCD}(a,b)=1$.
\end{cor}

Let $\mu(H)$ stand for the minimal number of generators of $H$. Hence $\mu(H) = \mu_R(P_0)$. Every $M \in \Max R$ is a $k$-rational closed point of $\Spec R$ if the base field $k$ is algebraically closed, so that by Theorem \ref{3.2} (2) and Corollary \ref{4.3} we readily get the following.

\begin{cor}[cf. Theorem \ref{fs}]\label{4.4}
Suppose that $k$ is an algebraically closed field. Then, $\mu_R(M) =2$ if $M \ne P_0$, and $\mu_R(M) = \mu(H)$ if $M = P_0$.
\end{cor}


As another application of Theorem \ref{4.2}, we have an explicit system of generators for the integral closures of  certain ideals in $R$. Let us  perform the task before closing this note. In what follows, let $I$ be an ideal of $R$. Similarly as in Section \ref{section3}, we write $IS = \varphi S$ with $\varphi \in S$, and assume that $\varphi = t^qf$, where $0 < q \in H$ and $f \in S$ with $f(0)=1$. Let $\fka = t^q S \cap R$. Hence, $\overline{I}= \fka{\cdot}(fS \cap R)$ by Lemma \ref{lem1}.  We set $n = \mu_R(\fka)$ and write $\fka= (t^{b_1}, t^{b_2}, \ldots, t^{b_n})$ with $b_i \in H$ such that $b_i \ge q$ for each $1 \le i \le n$.

We need  the following.

\begin{prop}
$n \ge 2$.
\end{prop}

\begin{proof}
Assume that $n = 1$. Then, $\fka = t^qR$, whence $\{h \in H \mid h \ge q\} \subseteq q + H$. Let $e = \min~[H \setminus \{0\}]$ and write $q = me + \alpha$ with integers $m,\alpha$ such that $m>0$ and $0 \le \alpha <e$. We then have $\alpha =0$. In fact, if $\alpha > 0$, then since $(m+1)e > q$ and $(m+1)e \in H$, we get $(m+1)e = q + h$ for some $h \in H$, so that $e = \alpha +h$. Hence,  $0 < h = e-\alpha <e$, which contradicts the minimality of $e$. Thus, $q = me$. Let $\Lambda = \{h \in H \mid q < h, h \not\equiv 0~\mod~e\}$. We set $q_1 = \min~\Lambda$ and write $q_1 = m_1 e + \alpha_1$ with integers $m_1, \alpha_1$ such that $0 < \alpha_1 < e$. Then, $m < m_1$. In fact, since $q = me < q_1$, we get $m \le m_1$. If $m = m_1$, then $q_1 = q + \alpha_1$, so that $\alpha_1 \in H$, since $q+\alpha_1 = q + h$ for some $h \in H$. This violates the minimality of $e$. Therefore, $m < m_1$, whence $q< q_1 - e$, because $q_1 -e = (m_1-1)e + \alpha_1 \ge me + \alpha_1 > q$. We now consider the fact that $q_1 = m_1e + \alpha_1 = q + h$ with $h \in H$. Then, since $q = me$, we have $(m-1)e + h= q + h -e = q_1 -e$, so that  $q_1 - e \in H$ and $q_1 - e\equiv q_1 \equiv \alpha_1 ~\mod~e$. Therefore, $q_1 - e \in \Lambda$, which still contradicts the minimality of $q_1$. Thus, $n \ne 1$. 
\end{proof}

Let us consider the specific case where $f = 1-t$. We set $\fkb = (t^{b_2}, t^{b_3},\ldots, t^{b_n})$. Then, $R/\fkb$ is an Artinian local ring with maximal ideal $P_0/\fkb$  (see Proposition \ref{core}), whence $t^mS \subseteq \fkb$ for all $m \gg 0$. Therefore, we can choose an integer $b \in H$ so that $t^b \in \fkb$ and $\operatorname{GCD}(b_1,b)=1$. Hence $b \ge q$.  Let $b_0 = b$. We then have the following.

\begin{thm}\label{4.7}
$\overline{I}= (t^{b_1}-t^{b_0}, \{t^{b_i} - t^{b_0+b_i}\}_{2 \le i \le n})$. Hence $\mu_R(\overline{I}) \le \mu_R(\fka)$.
\end{thm}

\begin{proof} Let $a = b_1$. By Lemma \ref{lem1} and Theorem \ref{4.2}, $\overline{I}= \fka{\cdot}(fS \cap R)= \fka{\cdot}(1-t^a, 1-t^b)$, while by Lemma \ref{lem1} $t^a-t^b \in \overline{I}$, because $t^a-t^b \in t^qS \cap R=\fka$ and $t^a - t^b \in fS\cap R$. Therefore, since $\fka = (t^{b_0}, t^{b_1},\ldots, t^{b_n})$, we get $$\overline{I} = (t^{b_i} - t^{a+b_i}\mid 0 \le i \le n)+(t^{b_i}- t^{b + b_i}\mid 0 \le i \le n).$$ Therefore, $\overline{I} = (t^{a} - t^{b})+(t^{b_i}- t^{b + b_i}\mid 0 \le i \le n)$, since
$$
t^{b_i}-t^{a+b_i}=(-t^{b_i})(t^a - t^b)+(t^{b_i} - t^{b+b_i})
$$
for each $0 \le i \le n$. We set $J = (t^{a} - t^{b})+(t^{b_i}- t^{b + b_i}\mid 2 \le i \le n)$. Then, since $t^b \in \fkb$, we have $b = b_i + h$ for some $2 \le i \le n$ and $h \in H$, so that $t^b- t^{b+b_0} =t^h(t^{b_i} - t^{b+b_i}) \in J$. On the other hand, because $$
t^{b_1} - t^{b+b_1} = (t^a -t^b) + (t^ {b_0} - t^{a+b_0}) \ \ \text{and}\ \ t^{b_0} - t^{a+b_0}=(-t^{b_0})(t^a - t^b) + (t^{b}- t^{b+b_0}),
$$
we get $t^{b_1}+ t^{b+b_1} \in J$. Thus $\overline{I} = J$, as claimed.
\end{proof}

Let us note one example.

\begin{ex}
Let $H=\left<4,11,13\right>$ and $R= k[H]$. Let $q =12$ and $f = 1-t$. Let $I$ be an ideal of $R$ such that $IS =( t^{12}f)S$. Hence, $\overline{I} = (t^{12}f)S \cap R$. We have $\fka = t^{12}S \cap R= (t^{12}, t^{13},t^{15}, t^{22})$. Set $\fkb=(t^{13},t^{15}, t^{22})$, and take $a = 12$, $b = 13$. Then, by Theorem \ref{4.7} $\overline{I}=(t^{12}-t^{13}, t^{13}-t^{26}, t^{15} - t^{28}, t^{22} - t^{35})$. We actually have that $\overline{I} = (t^{12}, t^{13}+t^{14}, t^{15}, t^{22})$. 
\end{ex}


\begin{thebibliography}{20}



\bibitem{EE2}
{\sc D. Eisenbud and E. G. Evans, Jr.}, Basic elements: Theorems from algebraic K-theory, {\em Bull. Amer. Math. Soc.}, {\bf 78} (1972), 546--549.

\bibitem{EE}
{\sc D. Eisenbud and E. G. Evans, Jr.}, Generating modules efficiently: theorems from algebraic K-theory, {\em J. Algebra}, {\bf 27} (1973), 278--305.




\bibitem{F}
{\sc O. Forster}, \"Uber die Anzahl der Erzeugenden eines Ideals in einem Noetherschen Ring, {\em Math. Z.}, {\bf 84} (1964), 80--87.



\bibitem{RG}
{\sc J. C. Rosales and P. A. Garc{\'i}a-S{\'a}nchez}, Numerical Semigroups, Developments in Mathematics, {\bf 20}, Springer, New York, 2009.


\bibitem{S}
{\sc J. D. Sally}, Numbers of generators of ideals in local rings. {\em Marcel Dekker, Inc.}, New York-Basel, 1978.

\bibitem{Swan}
{\sc R. G. Swan}, The number of generators of a module, {\em Math. Z.}, {\bf 102} (1967), 318--322.


\end{thebibliography}
\end{document}